\documentclass[12pt, reqno]{amsart}
\usepackage{amsmath, amsthm, amscd, amsfonts, amssymb, graphicx, color}
\usepackage[bookmarksnumbered, colorlinks, plainpages]{hyperref}

\input{mathrsfs.sty}
\newtheorem{theorem}{Theorem}[section]
\theoremstyle{plain}

\newtheorem{proposition}[theorem]{Proposition}
\newtheorem{corollary}[theorem]{Corollary}
\newtheorem{lemma}[theorem]{Lemma}
\newtheorem{remark}[theorem]{Remark}
\newtheorem{example}[theorem]{Example}
\numberwithin{equation}{section}

\begin{document}
\title{Cartesian decomposition and Numerical radius inequalities}
\author[F. Kittaneh, M.S. Moslehian, T. Yamazaki]{Fuad Kittaneh$^{1}$,
Mohammad Sal Moslehian$^{2}$ and Takeaki Yamazaki$^{3}$}
\address{$^{1}$Department of Mathematics, The University of Jordan, Amman,
Jordan}
\email{fkitt@ju.edu.jo}
\address{$^{2}$ Department of Pure Mathematics, Center Of Excellence in
Analysis on Algebraic Structures (CEAAS), Ferdowsi University of Mashhad, P.
O. Box 1159, Mashhad 91775, Iran.}
\email{moslehian@um.ac.ir}
\address{$^{3}$ Department of Electrical, Electronic and Computer
Engineering, Toyo University, Kawagoe-Shi, Saitama, 350-8585 Japan.}
\email{e-mail: t-yamazaki@toyo.jp}

\begin{abstract}
We show that if $T=H+iK$ is the Cartesian decomposition of $T\in \mathbb{B(\mathscr{H})}$, then for $\alpha ,\beta \in \mathbb{R}$, $\sup_{\alpha ^{2}+\beta
^{2}=1}\Vert \alpha H+\beta K\Vert =w(T)$. We then apply it to prove that if $A,B,X\in \mathbb{B(\mathscr{H})}$ and $0\leq mI\leq X$, then \begin{align*}
m\Vert \mbox{Re}(A)-\mbox{Re}(B)\Vert & \leq w(\mbox{Re}(A)X-X\mbox{Re}(B))
\\
& \leq \frac{1}{2}\sup_{\theta \in \mathbb{R}}\left\Vert (AX-XB)+e^{i\theta
}(XA-BX)\right\Vert \\
& \leq \frac{\Vert AX-XB\Vert +\Vert XA-BX\Vert }{2},
\end{align*}
where $\mbox{Re}(T)$ denotes the real part of an operator $T$. A refinement
of the triangle inequality is also shown.
\end{abstract}

\maketitle

\section{Introduction}

Let $\mathbb{B(\mathscr{H})}$ be the $C^{\ast }$-algebra of all bounded linear operators on a complex
Hilbert space $(\mathscr{H},\langle \cdot ,\cdot \rangle )$ and $I$ stand for the identity operator. If
$\dim \mathscr{H}=n$, we identify $\mathbb{B(\mathscr{H})}$ with the space $\mathcal{M}_{n}$ of all $n\times n$ matrices with entries
in the complex field and denote its identity by $I_{n}$. Any operator $T\in \mathbb{B(\mathscr{H})}$ can be represented as $T=H+iK$, the so-called Cartesian decomposition,
where $H=\mbox{Re}(T)=\frac{T+T^{\ast }}{2}$ and $K=\mbox{Im}(T)=\frac{%
T-T^{\ast }}{2i}$ are called the real and imaginary parts of $T$. An
operator $A\in\mathbb{B(\mathscr{H})}$ is called positive, denoted by $A\geq 0$, if $\langle Ax,x\rangle \geq 0$
for all $x\in\mathscr{H}$. For $p\geq 1$, the Schatten $p$-class, denoted by $\mathcal{C}_{p}$, is
defined to be the two-sided ideal in $\mathbb{B(\mathscr{H})}$ of those compact operators $A$ for which $\Vert A\Vert _{p}=\mbox{tr}%
(|A|^{p})^{1/p}$ is finite, where the symbol tr denotes the usual trace.
This norm as well as the usual operator norm $\Vert \cdot \Vert $ are
typical examples of unitarily invariant norms, i.e., a norm $|||\cdot |||$
defined on a two-sided ideal $\mathcal{C}_{||||\cdot |||}$ of $\mathbb{B(\mathscr{H})}$ satisfying $|||UAV|||=|||A|||$ for all $A\in \mathcal{C}_{||||\cdot |||}$
and all unitaries $U,V\in\mathbb{B(\mathscr{H})}$. The numerical radius of $A\in\mathbb{B(\mathscr{H})}$ is defined by
\begin{equation*}
w(A)=\sup \{|\langle Ax,x\rangle |:x\in\mathscr{H},\Vert x\Vert =1\}.
\end{equation*}
It is well known that $w(\cdot )$ defines a norm on $\mathbb{B(\mathscr{H})}$ such that for all $A\in\mathbb{B(\mathscr{H})}$,
\begin{equation}
\frac{1}{2}\Vert A\Vert \leq w(A)\leq \Vert A\Vert .  \label{1.1}
\end{equation}
If $A$ is self-adjoint, then $w(A)=\Vert A\Vert $ and if $A^{2}=0$, then $%
w(A)=\frac{\Vert A\Vert }{2}$(see e.g., \cite{7} and \cite{10}). Of course, $%
w(\cdot )$ is not unitarily invariant, rather it satisfies $w(U^{\ast
}AU)=w(A)$ for all $A$ and all unitary $U$ in $\mathbb{B(\mathscr{H})}$, i.e., $w(\cdot )$ is weakly unitarily invariant.

Some interesting numerical radius inequalities improving inequalities (\ref%
{1.1}) have been obtained by several mathematicians (see \cite{1}, \cite{6}, \cite{7},
\cite{15}, and references therein). Several investigations on norm and
numerical radius inequalities involving the Cartesian decomposition may be
found in the literature, among them we would like to refer the reader to
\cite{5} and \cite{8}.

In this note, we show that if $T=H+iK$ is the Cartesian decomposition of $%
T\in\mathbb{B(\mathscr{H})}$, then for $\alpha ,\beta \in \mathbb{R}$, $\sup_{\alpha ^{2}+\beta
^{2}=1}\Vert \alpha H+\beta K\Vert =w(T)$. We then apply it to find upper
and lower bounds for $w(\mbox{Re}(A)X-X\mbox{Re}(B))$, where $A,B,X\in\mathbb{B(\mathscr{H})}$ and $0\leq mI\leq X$. Furthermore, we present a refinement of the triangle
inequality.

\section{Results}

We start this section with a result concerning the Cartesian decomposition.

\begin{theorem}
\label{thm:1} Let $T=H+iK$ be the Cartesian decomposition of $T\in\mathbb{B(\mathscr{H})}$. Then for $\alpha ,\beta \in \mathbb{R}$,
\begin{equation}
\sup_{\alpha ^{2}+\beta ^{2}=1}\Vert \alpha H+\beta K\Vert =w(T).
\label{2.1}
\end{equation}
In particular,
\begin{equation}
\frac{1}{2}\Vert T+T^{\ast }\Vert \leq w(T)\quad \mbox{and}\quad \frac{1}{2}%
\Vert T-T^{\ast }\Vert \leq w(T).  \label{2.2}
\end{equation}
\end{theorem}

\begin{proof}
First of all, we note that
\begin{equation}
w(T)=\sup_{\theta \in \mathbb{R}}\Vert \mbox{Re}(e^{i\theta }T)\Vert \,.
\label{y3}
\end{equation}
In fact, $\sup_{\theta \in \mathbb{R}}\mbox{Re}\left( e^{i\theta }\langle
Tx,x\rangle \right) =|\langle Tx,x\rangle |$ yields that
\begin{equation*}
\sup_{\theta \in \mathbb{R}}\Vert \mbox{Re}(e^{i\theta }T)\Vert
=\sup_{\theta \in \mathbb{R}}w(\mbox{Re}(e^{i\theta }T))=w(T).
\end{equation*}

On the other hand, let $T=H+iK$ be the Cartesian decomposition of $T$. Then
\begin{align}
\mbox{Re}(e^{i\theta }T)& =\frac{e^{i\theta }T+e^{-i\theta }T^{\ast }}{2}=%
\frac{1}{2}\{(\cos \theta +i\sin \theta )T+(\cos \theta -i\sin \theta
)T^{\ast }\}  \notag \\
& =\left( \cos \theta \right) \frac{T+T^{\ast }}{2}-\left( \sin \theta
\right) \frac{T-T^{\ast }}{2i}=\left( \cos \theta \right) H-\left( \sin
\theta \right) K.  \label{2.4}
\end{align}%
Therefore, by putting $\alpha =\cos \theta $ and $\beta =-\sin \theta $ in (%
\ref{2.4}), we obtain (\ref{2.1}). Especially, by setting $(\alpha ,\beta
)=(1,0)$ and $(\alpha ,\beta )=(0,1)$, we reach (\ref{2.2}).
\end{proof}


\begin{remark}
By using $($\ref{2.2}$)$, we get some known inequalities:

\begin{itemize}
\item[(i)] $\|T\|=\|H+iK\|\leq \|H\|+\|K\| \leq 2w(T)$. Hence we have $\frac{%
1}{2}\|T\|\leq w(T)$.

\item[(ii)] If $T=T^{*}$, then $T=H$. Hence we have $\|T\|=\|H\|\leq
w(T)\leq \|T\|$ and $w(T)=\|T\|$.

\item[(iii)] By easy calculation, we have $\frac{T^{\ast }T+TT^{\ast }}{2}%
=H^{2}+K^{2}$. Hence,%
\begin{equation*}
\frac{1}{4}\Vert T^{\ast }T+TT^{\ast }\Vert =\frac{1}{2}\Vert
H^{2}+K^{2}\Vert \leq \frac{1}{2}(\Vert H\Vert ^{2}+\Vert K\Vert ^{2})\leq
w^{2}(T)\text{ \ }(\text{see also \cite{12}}).
\end{equation*}

\item[(iv)] Let $\alpha ,\beta \in \mathbb{R}$ satisfy $\alpha ^{2}+\beta
^{2}=1$. Then for any unit vector $x\in\mathscr{H}$, we have
\begin{align*}
\Vert (\alpha H+\beta K)x\Vert & =\left\Vert
\begin{bmatrix}
H & K \\
0 & 0%
\end{bmatrix}
\begin{bmatrix}
\alpha x \\
\beta x%
\end{bmatrix}
\right\Vert \leq \left\Vert
\begin{bmatrix}
H & K \\
0 & 0%
\end{bmatrix}
\right\Vert =\left\Vert
\begin{bmatrix}
H & K \\
0 & 0%
\end{bmatrix}
\begin{bmatrix}
H & 0 \\
K & 0%
\end{bmatrix}
\right\Vert ^{\frac{1}{2}} \\
& =\Vert H^{2}+K^{2}\Vert ^{\frac{1}{2}}=\frac{1}{\sqrt{2}}\Vert T^{\ast
}T+TT^{\ast }\Vert ^{\frac{1}{2}}.
\end{align*}%
Hence we have $w^{2}(T)=\sup_{\alpha ^{2}+\beta ^{2}=1}\Vert \alpha H+\beta
K\Vert ^{2}\leq \frac{1}{2}\Vert T^{\ast }T+TT^{\ast }\Vert $ $($see also
\cite{12}$)$.
\end{itemize}
\end{remark}

We can obtain a refinement of the triangle inequality as follows.


\begin{theorem}
\label{triangle inequality} Let $A,B\in\mathbb{B(\mathscr{H})}$. Then
\begin{equation*}
\Vert A+B\Vert \leq 2w\left(
\begin{bmatrix}
0 & A \\
B^{\ast } & 0%
\end{bmatrix}
\right) \leq \Vert A\Vert +\Vert B\Vert .
\end{equation*}
\end{theorem}

\begin{proof}
Let $T=%
\begin{bmatrix}
0 & A \\
B^{\ast } & 0%
\end{bmatrix}
$ on $\mathscr{H}\oplus\mathscr{H}$. Then by \eqref{2.2} and \eqref{y3}, we have
\begin{align*}
\Vert A+B\Vert & =\Vert T+T^{\ast }\Vert  \\
& \leq 2w(T) \\
& =\sup_{\theta \in \mathbb{R}}2\Vert \mbox{Re}(e^{i\theta }T)\Vert  \\
& =\sup_{\theta \in \mathbb{R}}\left\Vert
\begin{bmatrix}
0 & e^{i\theta }A+e^{-i\theta }B \\
e^{-i\theta }A^{\ast }+e^{i\theta }B^{\ast } & 0%
\end{bmatrix}
\right\Vert  \\
& =\sup_{\theta \in \mathbb{R}}\Vert e^{i\theta }A+e^{-i\theta }B\Vert \text{
(since }\left\Vert \left[
\begin{array}{cc}
0 & C \\
C^{\ast } & 0%
\end{array}
\right] \right\Vert =\left\Vert C\right\Vert \text{)} \\
& \leq \Vert A\Vert +\Vert B\Vert .
\end{align*}
\end{proof}

Thus we observe that equality occurs in the triangle inequality for the
operator norm if and only if the two equalities
\begin{equation*}
w\left(%
\begin{bmatrix}
0 & A \\
B^{*} & 0%
\end{bmatrix}
\right)+w\left(%
\begin{bmatrix}
0 & B \\
A^{*} & 0%
\end{bmatrix}
\right)=w\left(%
\begin{bmatrix}
0 & A+B \\
A^{*}+B^{*} & 0%
\end{bmatrix}
\right)
\end{equation*}
and
\begin{equation*}
w\left(%
\begin{bmatrix}
0 & A \\
B^{*} & 0%
\end{bmatrix}
\right)=w\left(%
\begin{bmatrix}
0 & A \\
0 & 0%
\end{bmatrix}
\right)+w\left(%
\begin{bmatrix}
0 & 0 \\
B^{*} & 0%
\end{bmatrix}
\right)
\end{equation*}
occur in the triangle inequality for the norm $w(\cdot)$.


\begin{example}
Let $A=%
\begin{bmatrix}
1 & 1 \\
0 & 1%
\end{bmatrix}
$, $B=%
\begin{bmatrix}
0 & -1 \\
0 & 0%
\end{bmatrix}
,$ and $T=%
\begin{bmatrix}
0 & A \\
B^{\ast } & 0%
\end{bmatrix}
$. Then $\Vert A+B\Vert <2w(T)<\Vert A\Vert +\Vert B\Vert .$

To see this, we need the fact \cite{3} that "if $A,B\in\mathbb{B(\mathscr{H})}$ are non-zero, then the equation $\Vert A+B\Vert =\Vert A\Vert +\Vert
B\Vert $ holds if and only if $\Vert A\Vert \Vert B\Vert \in \overline{%
W(A^{\ast }B)}$".

It is clear that $\Vert A+B\Vert =1$. Let $x=\frac{1}{2}[i\ 1\ 1\ 1]^{t}\in
\mathbb{C}^{4}$. Then $\Vert x\Vert =1$ and $|\langle Tx,x\rangle |=\frac{%
\sqrt{10}}{4}$. Hence $\Vert A+B\Vert =1<\frac{\sqrt{10}}{2}\leq 2w(T).$ On
the other hand, assume that $2w(T)=\Vert A\Vert +\Vert B\Vert $. Then by
Theorem \ref{triangle inequality}, there exists $\theta \in \mathbb{R}$ such
that $\Vert e^{i\theta }A+e^{-i\theta }B\Vert =\Vert A\Vert +\Vert B\Vert $.
So, by the above fact, we have $\Vert A\Vert \Vert B\Vert \in \overline{%
W(e^{-2i\theta }A^{\ast }B)}$, and it is equivalent to $e^{2i\theta }\Vert
A\Vert \Vert B\Vert \in \overline{W(A^{\ast }B)}$. Since $w(A^{\ast }B)\leq
\Vert A^{\ast }B\Vert \leq \Vert A\Vert \Vert B\Vert $, we have $w(A^{\ast
}B)=\Vert A\Vert \Vert B\Vert $. However, $\Vert A\Vert \Vert B\Vert =\frac{%
3+\sqrt{5}}{2}$ and
\begin{equation*}
w(A^{\ast }B)=w\left(
\begin{bmatrix}
0 & -1 \\
0 & -1%
\end{bmatrix}
\right) =\frac{1+\sqrt{2}}{2}
\end{equation*}
$($see e.g., \cite[\textit{Example} 3\textit{\ in Section} 2.5.1]{9}$).$
Hence $w(A^{\ast }B)<\Vert A\Vert \Vert B\Vert $, which leads to a
contradiction. Hence $2w(T)=\sup_{\theta \in \mathbb{R}}\Vert e^{i\theta
}A+e^{-i\theta }B\Vert <\Vert A\Vert +\Vert B\Vert ,$ and so the
inequalities in Theorem \ref{triangle inequality} can be strict.
\end{example}

The following lemma is known in the literature.

\begin{lemma}
\cite[Lemma 3.1]{14}\label{vA} Let $X\geq mI>0$ for some positive real
number $m$ and $Y$ be in the associate ideal corresponding to a unitarily
invariant norm $|||\cdot |||$. Then
\begin{equation*}
m|||Y|||\leq \frac{1}{2}|||YX+XY|||\,.
\end{equation*}
\end{lemma}

The next assertion is interesting on its own right (see also \cite{4}).

\begin{proposition}
\label{l1}Let $A,B\in\mathbb{B(\mathscr{H})}$ be self-adjoint and $0<mI\leq X$ for some positive real number $m$. Then
\begin{equation}
m\Vert A-B\Vert \leq w(AX-XB)\leq \Vert AX-XB\Vert .  \label{2.5}
\end{equation}
\end{proposition}

\begin{proof}
Let $T=AX-XB$. Then $T+T^{\ast }=(A-B)X+X(A-B)$. It follows from Lemma \ref%
{vA} that
\begin{equation*}
m\Vert A-B\Vert \leq \frac{1}{2}\Vert (A-B)X+X(A-B)\Vert =\frac{1}{2}\Vert
T+T^{\ast }\Vert \leq w(T)=w(AX-XB).
\end{equation*}
The second inequality of (\ref{2.5}) follows from \eqref{1.1}.
\end{proof}

Proposition \ref{l1} improves Theorem 1 in \cite{4} for the usual operator
norm, which says that $m\Vert A-B\Vert \leq \Vert AX-XB\Vert $.

In the setting of matrices, it is known that for $A\in \mathcal{M}_{n},$ we
have%
\begin{equation*}
\Vert A\Vert \leq \Vert A\Vert _{p},
\end{equation*}
and so%
\begin{equation*}
w(A)\leq \Vert A\Vert _{p}\,.
\end{equation*}
Using (\ref{2.5}) and the fact that $\Vert A\Vert _{p}\leq \Vert A\Vert
\,\Vert I_{n}\Vert _{p},$ we infer the next result.

\begin{proposition}
Let $A,B,X\in \mathbb{M}_{n}$ be Hermitian and $0<mI_{n}\leq X$ for some
positive real number $m$. Then $\frac{m}{n^{1/p}}\Vert A-B\Vert _{p}\leq
w(AX-XB)\leq \Vert AX-XB\Vert _{p}\,.$
\end{proposition}

An extension of Proposition \ref{l1} to arbitrary (i.e., not necessarily
self-adjoint) operators $A$, $B$ can be stated as follows.

\begin{theorem}
\label{T2.8}Let $A,B\in\mathbb{B(\mathscr{H})}$ and $0<mI\leq X$ for some positive real number $m$. Then
\begin{equation}
m\Vert A-B\Vert \leq w\left( \left[
\begin{array}{cc}
0 & AX-XB \\
A^{\ast }X-XB^{\ast } & 0%
\end{array}
\right] \right) \leq \frac{\Vert AX-XB\Vert +\Vert A^{\ast }X-XB^{\ast
}\Vert }{2}.  \label{2.6}
\end{equation}
\end{theorem}

\begin{proof}
Applying Proposition \ref{l1} to the self-adjoint operators $\hat{A}=\left[
\begin{array}{cc}
0 & A \\
A^{\ast } & 0%
\end{array}
\right] ,$ $\hat{B}=\left[
\begin{array}{cc}
0 & B \\
B^{\ast } & 0%
\end{array}
\right] $, and the positive operator $\hat{X}=\left[
\begin{array}{cc}
X & 0 \\
0 & X%
\end{array}
\right] $ on $\mathscr{H}\mathcal{\oplus }\mathscr{H}$, we have%
\begin{equation*}
m\Vert \hat{A}-\hat{B}\Vert \leq w\left( \hat{A}\hat{X}-\hat{X}\hat{B}%
\right) .
\end{equation*}
Thus,%
\begin{equation*}
m\Vert A-B\Vert \leq w\left( \left[
\begin{array}{cc}
0 & AX-XB \\
A^{\ast }X-XB^{\ast } & 0%
\end{array}
\right] \right) .
\end{equation*}
This proves the first inequality in (\ref{2.6}). To prove the second
inequality in (\ref{2.6}), use the triangle inequality for $w(\cdot )$ and
the fact that $w\left( \left[
\begin{array}{cc}
0 & AX-XB \\
0 & 0%
\end{array}
\right] \right) =\frac{\Vert AX-XB\Vert }{2}$ and $w\left( \left[
\begin{array}{cc}
0 & 0 \\
A^{\ast }X-XB^{\ast } & 0%
\end{array}
\right] \right) =\frac{\Vert A^{\ast }X-XB^{\ast }\Vert }{2}.$
\end{proof}

It should be mentioned here that (\ref{2.5}) follows as a special case of (%
\ref{2.6}) by recalling that $w\left( \left[
\begin{array}{cc}
0 & C \\
C & 0%
\end{array}
\right] \right) =w(C)$ (see, e.g., \cite{11}).

It follows from Corollary 5 in \cite{4} that if $U, V\in\mathbb{B(\mathscr{H})}$ are unitary and $0<mI\leq X$ for some positive real number $m$, then%
\begin{equation}
m\Vert U-V\Vert \leq \left\Vert UX-XV\right\Vert .  \label{2.7}
\end{equation}
A refinement of (\ref{2.7}) is given in the following corollary.

\begin{corollary}
\label{C2.4}Let $U,V\in\mathbb{B(\mathscr{H})}$ be unitary and $0<mI\leq X$ for some positive real number $m$. Then%
\begin{equation}
m\Vert U-V\Vert \leq w\left( \left[
\begin{array}{cc}
0 & UX-XV \\
U^{\ast }X-XV^{\ast } & 0%
\end{array}
\right] \right) \leq \left\Vert UX-XV\right\Vert .  \label{2.8}
\end{equation}
\end{corollary}

\begin{proof}
The second inequality in (\ref{2.8}) follows from by the unitary invariance
of $\left\Vert \cdot \right\Vert $. In fact, $\Vert U^{\ast }X-XV^{\ast
}\Vert =\left\Vert U^{\ast }\left( XV-UX\right) V^{\ast }\right\Vert
=\left\Vert XV-UX\right\Vert =\left\Vert UX-XV\right\Vert $.
\end{proof}

Finally, we present another extension of Proposition \ref{l1} to arbitrary
operators $A,B$. To achieve it, we need the following lemma. It immediately
follows from the relations
\begin{eqnarray*}
w(X+Y) &=&w\left(
\begin{bmatrix}
0 & X+Y \\
X+Y & 0%
\end{bmatrix}
\right) \leq w\left(
\begin{bmatrix}
0 & X \\
Y & 0%
\end{bmatrix}
\right) +w\left(
\begin{bmatrix}
0 & Y \\
X & 0%
\end{bmatrix}
\right) \\
&=&2w\left(
\begin{bmatrix}
0 & X \\
Y & 0%
\end{bmatrix}
\right) .
\end{eqnarray*}

\begin{lemma}
\cite[Theorem 2.4]{11}\label{nneeww} If $X,Y\in
\mathbb{B(\mathscr{H})}$, then $w(X+Y)\leq 2w\left(
\begin{bmatrix}
0 & X \\
Y & 0%
\end{bmatrix}
\right) $.
\end{lemma}

\begin{theorem}
\label{main} Let $A,B\in\mathbb{B(\mathscr{H})}$ and $0<mI\leq X$ for some positive real number $m$. Then
\begin{align*}
m\Vert \mbox{Re}(A)-\mbox{Re}(B)\Vert & \leq w(\mbox{Re}(A)X-X\mbox{Re}(B))
\\
& \leq \frac{1}{2}\sup_{\theta \in \mathbb{R}}\left\Vert (AX-XB)+e^{i\theta
}(XA-BX)\right\Vert \\
& \leq \frac{\Vert AX-XB\Vert +\Vert XA-BX\Vert }{2}\,.
\end{align*}
\end{theorem}

\begin{proof}
We have
\begin{align*}
& m\Vert \mbox{Re}(A)-\mbox{Re}(B)\Vert \\
& \leq w(\mbox{Re}(A)X-X\mbox{Re}(B))\qquad \qquad \qquad \qquad \qquad
\qquad \qquad \quad (\mbox{by Proposition~}\ref{l1}) \\
& =\frac{w\left( (AX-XB)+(A^{\ast }X-XB^{\ast })\right) }{2} \\
& \leq w\left(
\begin{bmatrix}
0 & AX-XB \\
A^{\ast }X-XB^{\ast } & 0%
\end{bmatrix}
\right) \qquad \qquad \qquad \qquad \qquad \qquad (\mbox{by Lemma~}\ref%
{nneeww}) \\
& =\sup_{\theta \in \mathbb{R}}\left\Vert \mbox{Re}\left(
\begin{bmatrix}
0 & e^{i\theta }(AX-XB) \\
e^{i\theta }(A^{\ast }X-XB^{\ast }) & 0%
\end{bmatrix}
\right) \right\Vert \qquad \qquad \qquad \qquad (\mbox{by~}\eqref{y3}) \\
& ={\tiny \frac{1}{2}\sup_{\theta \in \mathbb{R}}\left\Vert
\begin{bmatrix}
0 & e^{i\theta }[(AX-XB)+e^{-2i\theta }(A^{\ast }X-XB^{\ast })^{\ast }] \\
e^{-i\theta }[(AX-XB)+e^{-2i\theta }(A^{\ast }X-XB^{\ast })^{\ast }]^{\ast }
& 0%
\end{bmatrix}
\right\Vert } \\
& =\frac{1}{2}\sup_{\theta \in \mathbb{R}}\left\Vert (AX-XB)+e^{-2i\theta
}(XA-BX)\right\Vert \\
& =\frac{1}{2}\sup_{\theta \in \mathbb{R}}\left\Vert (AX-XB)+e^{i\theta
}(XA-BX)\right\Vert \\
& \leq \frac{\Vert AX-XB\Vert +\Vert XA-BX\Vert }{2}\,.
\end{align*}
\end{proof}
 
\medskip


\end{document}